\documentclass[12pt,letterpaper]{article}
\usepackage{amsthm,amssymb,amsmath}
\usepackage[pdftex,colorlinks=true]{hyperref}
\usepackage{graphicx}

\theoremstyle{plain}
\newtheorem{lemma}{Lemma}[section]           
\newtheorem{theorem}[lemma]{Theorem}         
\newtheorem{conjecture}[lemma]{Conjecture}   
\newtheorem{question}[lemma]{Question}       

\theoremstyle{definition}

\newcommand{\ggcqedsymbol}{$\square$}
\newcommand{\ggcqed}{\hbox{}\nobreak\hbox{\quad\ggcqedsymbol}}
\newcommand{\ggcendpf}{\ggcqed}
\newcommand{\ggcnopf}{\ggcqed}

\newcommand{\defterm}[1]{\emph{#1}} 
\newcommand{\abstdefterm}[1]{#1} 

\newcommand{\doi}[1]{\href{http://dx.doi.org/#1}{\texttt{doi:#1}}.}
\newcommand{\mr}[1]{\href{https://mathscinet.ams.org/mathscinet-getitem?mr=MR#1}{MR#1}.}

\date{May 19, 2022}

\title{List Multicoloring of Planar Graphs and Related Classes}

\author{Glenn G.~Chappell\\
\small Department of Computer Science\\
\small University of Alaska\\
\small Fairbanks, AK 99775-6670\\
\small\texttt{ggchappell{@}alaska.edu}}

\begin{document}

\maketitle
\centerline{\small \textit{2020 Mathematics Subject Classification.}
 Primary 05C15; Secondary 05C10, 05C83.}
\centerline{\small \textit{Key words and phrases.}
  List coloring, choosability, planar graph, graph minor.}

\begin{abstract}
For positive integers $a$ and $b$,
a graph $G$ is \abstdefterm{$(a:b)$-choosable} if,
for each assignment of lists of $a$ colors
to the vertices of $G,$
each vertex can be colored with a set of $b$
colors from its list
so that adjacent vertices are colored with disjoint sets.

We show that
for positive integers $a$ and $b$,
every bipartite planar graph is $(a:b)$-choosable
iff $\frac{a}{b} \ge 3$.
For general planar graphs,
we show that if $\frac{a}{b} < 4\frac{2}{5}$,
then there exists a planar graph
that is not $(a:b)$-choosable,
thus improving on a result of X.~Zhu,
which had $4\frac{2}{9}$.
Lastly,
we show that
every $K_5$-minor-free graph is $(a:b)$-choosable
iff $\frac{a}{b} \ge 5$.
Along the way, we mention some open problems.
\end{abstract}

\section{Introduction} \label{S:intro}

For $L$ an assignment of lists of colors to
the vertices of a (finite, simple) graph $G,$
a \defterm{$b$-fold $L$-coloring} of $G$
is a mapping $\varphi$ that colors each vertex $v$ of $G$
with a set $\varphi(v)\subseteq L(v)$ of $b$ colors
such that adjacent vertices are colored with disjoint sets.

Following Erd\H{o}s, Rubin, and Taylor~\cite[p.~155]{ERT1980},
for positive integers $a$ and $b$
we say a graph $G$ is
\defterm{$(a:b)$-choosable} if,
for each assignment $L$ of colors
with $|L(v)| = a$ for each vertex $v$,
graph $G$ admits a $b$-fold $L$-coloring.
So the usual notion of $k$-choosability
is the same as $(k:1)$-choosability.

We are interested in results of the following form.
For some fixed class of graphs,
the following are equivalent for positive integers $a$ and $b$:
(i) every graph in the class is $(a:b)$-choosable;
(ii) $\frac{a}{b} \ge r$
(where $r$ is a number that depends on the class of interest).

\medskip

In Section~\ref{S:bipplanar}
we consider the class of bipartite planar graphs.
Alon \& Tarsi~\cite[Corollary~3.4]{AlTa1992} showed that every
bipartite planar graph is $3$-choosable.
The following stronger result was proven by
Gutner~\cite[Corollary~1.11]{Gut1992}
(see also Gutner \& Tarsi~\cite[Corollary~1.11]{GuTa2009}).

\begin{theorem}[Gutner 1992]
  \label{T:planarbip3m-m}
Let $G$ be a bipartite planar graph.
Then $G$ is $(3m:m)$-choosable,
for each positive integer $m$.\ggcnopf\end{theorem}

Using examples based on a construction of
Bar\'{a}t, Joret, \& Wood~\cite[proof of Thm.~1]{BJW2011},
we improve on Gutner's result, showing that
every bipartite planar graph is $(a:b)$-choosable
if and only if
$\frac{a}{b} \ge 3$.
Thus we have a result in our desired form.

\medskip

In Section~\ref{S:planar}
we consider planar graphs in general.
Erd\H{o}s, Rubin, \& Taylor~\cite[p.~153]{ERT1980}
conjectured that every planar graph is $5$-choosable,
while there exists a planar graph that is not $4$-choosable.

That every planar graph is $5$-choosable was proven by
Thomassen~\cite{Tho1994}.
Tuza \& Voigt~\cite[Thm.~3.1]{TuVo1996}
generalized Thomassen's argument to prove
the following.

\begin{theorem}[Tuza \& Voigt 1996]
  \label{T:planar5m-m}
Let $G$ be a planar graph.
Then $G$ is $(5m:m)$-choosable,
for each positive integer $m$.\ggcnopf\end{theorem}

A construction of
a planar graph that is not $4$-choosable
was given by
Voigt~\cite{Voi1993}.
Smaller examples are due to
Gutner~\cite[Thm.~1.7]{Gut1996}
and to
Mirzakhani~\cite{Mir1996}.
Using a method similar to that of Gutner,
one can construct, for each positive integer $m$,
a planar graph that is not $(4m:m)$-choosable;
for example, see Zhu~\cite[Thm.~1]{Zhu2017}.

So if $\frac{a}{b} \ge 5$,
then every planar graph is $(a:b)$-choosable,
while if $\frac{a}{b} \le 4$,
then there exists a planar graph that is not $(a:b)$-choosable.
What about ratios strictly between $4$ and $5$?
Zhu~\cite[Thm.~1]{Zhu2017} proved the following.

\begin{theorem}[Zhu 2017]
  \label{T:planar38-9}
Let $a$ and $b$ be positive integers.
If $\frac{a}{b} < 4\frac{2}{9}$,
then there exists a planar graph
that is not $(a:b)$-choosable.\ggcnopf\end{theorem}

We improve on Zhu's result by showing that
if $\frac{a}{b} < \frac{22}{5} = 4\frac{2}{5}$,
then there exists a planar graph that is not $(a:b)$-choosable.
However,
we do not have a result in our desired form.
We speculate on whether such a result
might hold.

\medskip

In Section~\ref{S:k5minfree}
we consider the larger class of $K_5$-minor-free graphs.
\v{S}krekovski~\cite[Thm.~2.3]{Skr1998}
generalized Thomassen's proof of the $5$-choosability
of planar graphs, showing the following.

\begin{theorem}[\v{S}krekovsi 1998]
  \label{T:k5free-5}
Let $G$ be a $K_5$-minor-free graph.
Then $G$ is $5$-choosable.\ggcnopf
\end{theorem}

Other proofs of \v{S}krekovski's result are due to
He, Miao, \& Shen~\cite[Thm.~2.1]{HMS2008}
and to
Wood \& Linusson~\cite[Thm.~1]{WoLi2010}.

We prove a result in our desired form
that generalizes both the
Tuza-Voigt result (Theorem~\ref{T:planar5m-m})
and the \v{S}krekovski result (Theorem~\ref{T:k5free-5}):
that every $K_5$-minor-free graph is $(a:b)$-choosable
if and only if
$\frac{a}{b} \ge 5$.

\medskip

We denote the vertex set of a graph $G$ by $V(G)$.
When describing lists of colors,
we will generaly omit union operators in a union of disjoint sets.
For example, $\mathit{XPT}$ means $X\cup P\cup T,$
for disjoint sets $X,$ $P,$ and $T.$

\section{Bipartite Planar Graphs} \label{S:bipplanar}

In this section we prove the following theorem.

\begin{theorem}
  \label{T:planarbip3m-m-iff}
The following are equivalent
for positive integers $a$ and $b$.
\begin{enumerate}
\item Every bipartite planar graph is $(a:b)$-choosable.
\item $\displaystyle\frac{a}{b}\ge 3$.\ggcnopf
\end{enumerate}
\end{theorem}

We begin with a lemma giving a list-coloring property
of $P_4$, a $4$-vertex path
(see Figure~\ref{Fig:path}).
Later, we will verify the (i) $\Longrightarrow$ (ii)
portion of Theorem~\ref{T:planarbip3m-m-iff}
using examples constructed by pasting together
multiple copies of $P_4$.

\begin{figure}  
\centering
\includegraphics{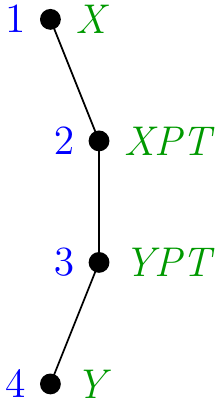}
\caption{$P_4$, a $4$-vertex path,
with vertices labeled
and lists of colors shown,
used in Lemma~\ref{L:planarbip-piece}
and the proof of Theorem~\ref{T:planarbip3m-m-iff}.}
\label{Fig:path}
\end{figure}

\begin{lemma}
  \label{L:planarbip-piece}
Let $a$ and $b$ be positive integers with
$2\le \frac{a}{b} < 3$.
Let $X,$ $Y,$ $P,$ and $T$ be pairwise disjoint
lists of colors,
such that $X,$ $Y,$ and $P$ have size $b$,
while $T$ has size $a-2b$.
(Note that $a-2b \ge 0$;
if $\frac{a}{b} = 2$, then $T = \varnothing$.)

Define a color assignment $L$
for $P_4$, a $4$-vertex path, as follows.
Label the vertices $1$, $2$, $3$, $4$,
in order along the path.
Let $L(1) = X,$
$L(2) = \mathit{XPT},$
$L(3) = \mathit{YPT},$
and $L(4) = Y,$
as shown in Figure~\ref{Fig:path}.

Then $P_4$ admits no $b$-fold $L$-coloring.
\end{lemma}

\begin{proof}
In a $b$-fold $L$-coloring $\varphi$ of $P_4$,
we must have $\varphi(1) = X$
and
$\varphi(4) = Y.$
So $\varphi(2)$ and $\varphi(3)$
are disjoint subsets of $\mathit{PT},$
each of size $b$,
and thus $|\varphi(2)\cup\varphi(3)| = 2b$.
However, $|\mathit{PT}| < 2b$,
so no such coloring can exist.\ggcendpf
\end{proof}

When considering $(a:b)$-choosability,
Lemma~\ref{L:planarbip-piece}
says we can forbid a specific
coloring of vertices $1$ and $4$ of $P_4$
with disjoint sets of colors.
Using this idea,
we can paste together copies of $P_4$
to construct a bipartite planar graph in which
all possible colorings of two vertices are forbidden.
We use this idea in the proof of Theorem~\ref{T:planarbip3m-m-iff}.

\begin{proof}[Proof of Theorem~\ref{T:planarbip3m-m-iff}]
(ii) $\Longrightarrow$ (i).
This follows from Theorem~\ref{T:planarbip3m-m}.

\medskip

\noindent
(i) $\Longrightarrow$ (ii).
Let $a$ and $b$ be positive integers
with $\frac{a}{b} < 3$.
We construct a bipartite planar graph $G$
such that $G$ is not $(a:b)$-choosable.

If $\frac{a}{b} < 2$, then we may let $G = K_2$.
Suppose, therefore, that $2 \le \frac{a}{b} < 3$.
Let
\[
q = \binom{a}{b} \binom{a-b}{b}.
\]

To construct graph $G,$
begin with $q$ copies of $P_4$,
pictured in Figure~\ref{Fig:path}.
Identify all the $1$ vertices in these copies,
labeling the resulting vertex as $1$.
Similarly identify all the $4$ vertices,
labeling the resulting vertex as $4$.
Add an edge joining vertices $1$ and $4$.
Let $G$ be the resulting graph.
See Figure~\ref{Fig:multipath} for an illustration
of graph $G.$
(The construction of $G$
is a variation on a construction of
Bar\'{a}t, Joret, \& Wood~\cite[proof of Thm.~1]{BJW2011}.)

\begin{figure}  
\centering
\includegraphics{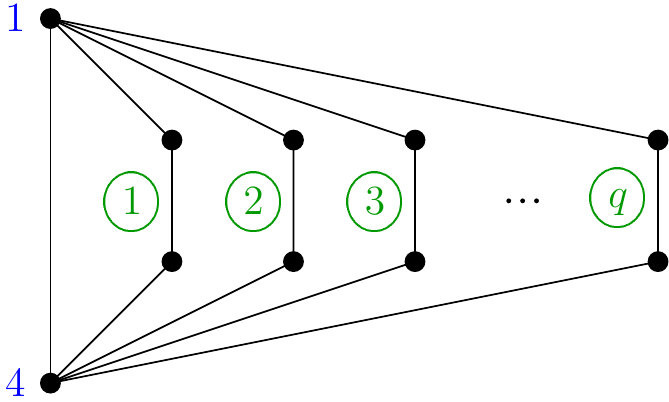}
\caption{Graph $G$ from the (i) $\Longrightarrow$ (ii) part
of the proof of Theorem~\ref{T:planarbip3m-m-iff},
with vertices $1$ and $4$ labeled.
Each copy of $P_4$
is labeled with a circled number from $1$ to $q$.}
\label{Fig:multipath}
\end{figure}

Graph $G$ is bipartite and planar.
It remains to show that $G$ is not $(a:b)$-choosable.

Assign vertices $1$ and $4$ the same list of $a$ colors.
The number of ways these two vertices can be colored with
disjoint sets of size $b$ is $\binom{a}{b}\binom{a-b}{b} = q$.
Create an (arbitrary) correspondence between these colorings
and the $q$ copies of $P_4$.
For each possible coloring of vertices $1$ and $4$,
assign lists of colors to the $2$ and $3$
vertices in the corresponding copy of $P_4$
so that Lemma~\ref{L:planarbip-piece}
allows us to conclude that vertices $1$ and $4$
cannot be colored in this manner.

The result is an assignment $L$ of lists of colors
with $|L(v)| = a$ for all $v\in V(G)$,
such that $G$ admits no $b$-fold $L$-coloring,
since no coloring is possible for vertices $1$ and $4$.
Thus, $G$ is not $(a:b)$-choosable.\ggcendpf\end{proof}

\section{General Planar Graphs} \label{S:planar}

In this section we prove the following theorem.

\begin{theorem}
  \label{T:planar22-5}
Let $a$ and $b$ be positive integers.
If $\frac{a}{b} < \frac{22}{5}$,
then there exists a planar graph $G$
such that $G$ is not $(a:b)$-choosable.\ggcnopf
\end{theorem}

Once again,
our proof will use
examples constructed by pasting together
small graphs.
Our construction is somewhat similar
to one due to Zhu~\cite[Lemma~2, proof of Thm.~1]{Zhu2017}---which,
in turn, has similarities
with a construction of Gutner~\cite[proof of Thm.~1.7]{Gut1996}.

\begin{figure}  
\centering
\includegraphics{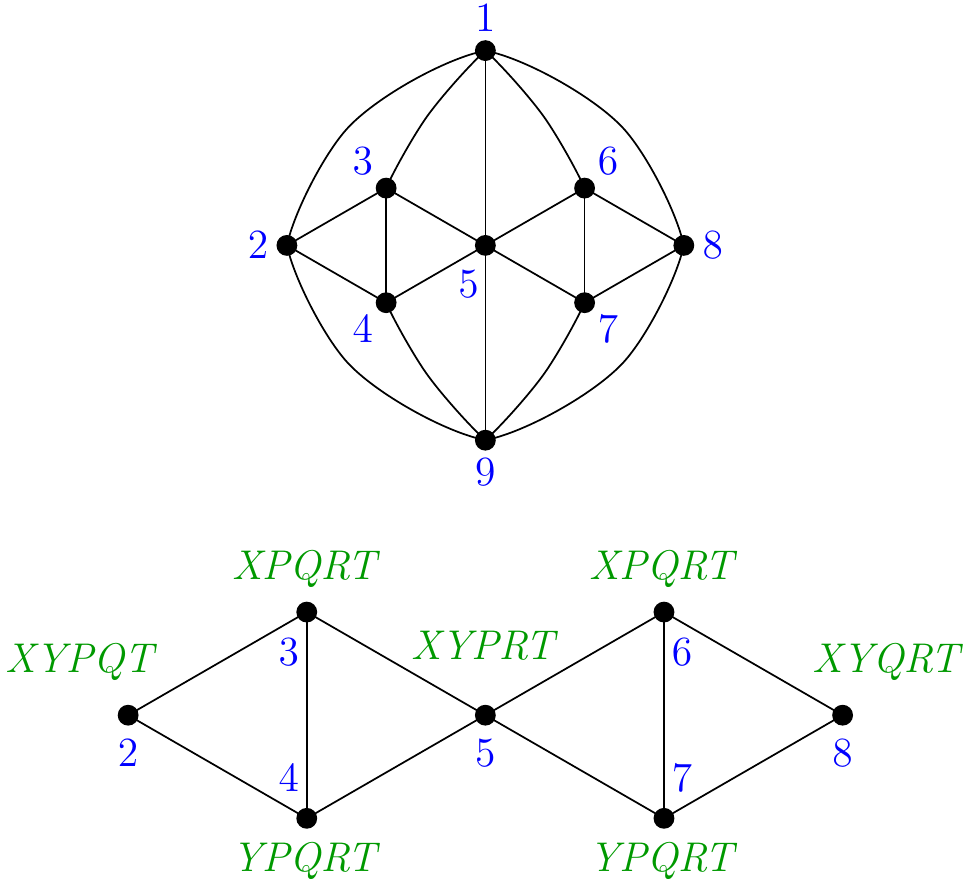}
\caption{Above, graph $F_1$ from Lemmas~\ref{L:planar-piece1}
and \ref{L:planar-piece2}.
Below, detail of $F_1$
showing lists of colors assigned to vertices $2$ through $8$.
Graph $F_1$ is identical to a graph
of Gutner~\cite[Figure~1]{Gut1996}
and of Voigt \& Wirth~\cite[Figure~1]{VoWi1997};
however, our color lists differ.}
\label{Fig:diamond2}
\end{figure}

We begin with two lemmas concerning graphs we call $F_1$ and $F_2$,
which are pictured in
the top of Figure~\ref{Fig:diamond2}
and in Figure~\ref{Fig:diamond4},
respectively.

\begin{lemma}
  \label{L:planar-piece1}
Let $a$ and $b$ be positive integers
with $4 \le \frac{a}{b} < \frac{22}{5}$.
Let $X,$ $Y,$ $P,$ $Q,$ $R,$ and $T$ be pairwise disjoint
lists of colors,
such that $X,$ $Y,$ $P,$ $Q,$ and $R$ have size $b$,
while $T$ has size $a-4b$.
(Note that $a-4b \ge 0$;
if $\frac{a}{b} = 4$, then $T = \varnothing$.)

Let $F_1$ be the graph pictured in the top portion of
Figure~\ref{Fig:diamond2},
with vertices labeled $1,\dots,9$ as shown.
Define a color assignment $L$
such that $L(1) = X$ and $L(9) = Y,$
while
vertices $2$ through $8$ are assigned lists of colors
as shown in the bottom portion of
Figure~\ref{Fig:diamond2}.

Then
in any $b$-fold $L$-coloring $\varphi$ of $F_1$,
we have $|\varphi(8)\cap T| > \frac{1}{2} |T|$;
that is,
the set with which vertex $8$ is colored
includes more than half of the elements of $T.$
\end{lemma}

\begin{proof}
Suppose that $\varphi$ is a $b$-fold $L$-coloring of $F_1$.
Observe that $\varphi(v) \cap \mathit{XY} = \varnothing$
for all $v\in \{2,\dots,8\}$.

We begin by proving two claims.

\medskip
\noindent\textbf{Claim~1.}
$|\varphi(2)\cap \varphi(5)| \ge 5b-a$.
\medskip

Suppose not:
$|\varphi(2)\cap \varphi(5)| < 5b-a$.
Then
\begin{align*}
|\varphi(2)\cup \varphi(5)|
  &= |\varphi(2)| + |\varphi(5)| - |\varphi(2)\cap \varphi(5)|\\
  &> b + b - (5b-a)\\
  &= a-3b.
\end{align*}
Colors usable on vertices $3$ and $4$
are those in $\mathit{PQRT}$:
a total of $3b+(a-4b) = a-b$ colors.
Removing colors in $\varphi(2)\cup \varphi(5)$---more
than $a-3b$ colors, by the above---the
number of colors still available for vertices $3$ and $4$
is less than $(a-b)-(a-3b) = 2b$.
But vertices $3$ and $4$ are adjacent,
so $|\varphi(3)\cup \varphi(4)| = 2b$,
a contradiction,
and Claim~1 is proven.

\medskip
\noindent\textbf{Claim~2.}
$|\varphi(5)\cap \varphi(8)| \ge 5b-a$.
\medskip

The proof of Claim~2 is much the same as that of Claim~1;
only the vertex labels differ.

\medskip
\noindent\textbf{Finishing.}
Now we prove the conclusion of the lemma, that
\[
|\varphi(8)\cap T| > \frac{1}{2}|T|.
\]

Suppose not:
$|\varphi(8)\cap T| \le \frac{1}{2} |T| = \frac{a-4b}{2}$.
Then
\begin{align*}
10b-2a
  &\le |\varphi(2)\cap \varphi(5)|+|\varphi(5)\cap \varphi(8)|
    && \text{by Claims~1 \& 2}\\
  &\le \Bigl(|\varphi(5)\cap P| + |\varphi(5)\cap T|\Bigr) &&\\
  &\qquad\qquad\qquad + \Bigl(|\varphi(5)\cap R| + |\varphi(8)\cap T|\Bigr) &&\\
  &= |\varphi(5)| + |\varphi(8)\cap T| &&\\
  &\le b + \frac{a-4b}{2} &&\text{by our supposition}\\
  &= \frac{a}{2}-b. &&
\intertext{Gathering terms in $10b-2a\le \frac{a}{2}-b$, we have:}
11b &\le \frac{5}{2}a &&\\
  &< \frac{5}{2}\cdot\frac{22}{5} b = 11b
    &&\text{since $\frac{a}{b} < \frac{22}{5}$.}
\end{align*}
But $11b < 11b$ is impossible,
and the lemma is proven.\ggcendpf\end{proof}

\begin{figure}  
\centering
\includegraphics{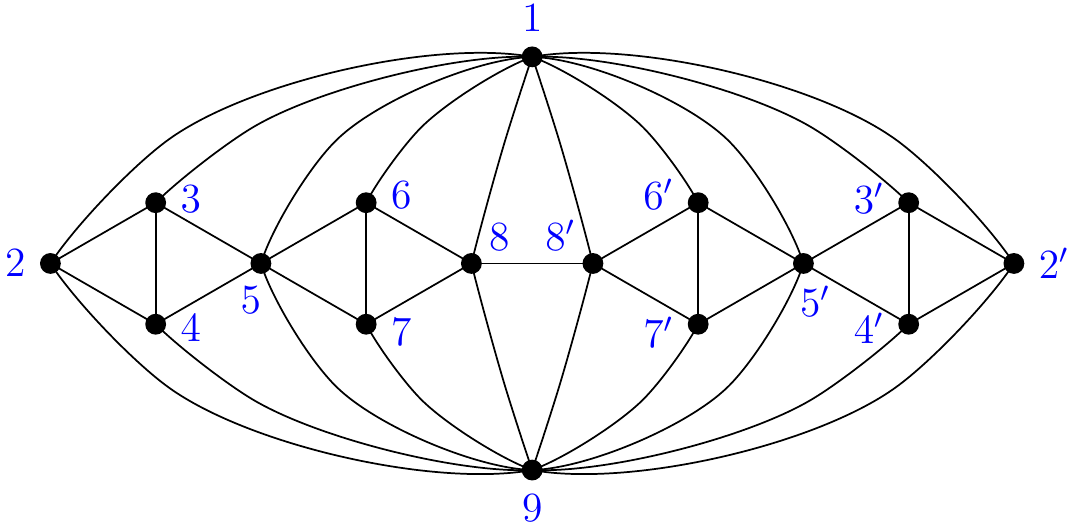}
\caption{Graph $F_2$ from Lemma~\ref{L:planar-piece2}
and the proof of Theorem~\ref{T:planar22-5}.}
\label{Fig:diamond4}
\end{figure}

\begin{lemma}
  \label{L:planar-piece2}
Let $a$ and $b$ be positive integers
with $4 \le \frac{a}{b} < \frac{22}{5}$.
Let $X,$ $Y,$ $P,$ $Q,$ $R,$ and $T$ be pairwise disjoint
lists of colors,
such that $X,$ $Y,$ $P,$ $Q,$ and $R$ have size $b$,
while $T$ has size $a-4b$.
(Note that $a-4b \ge 0$;
if $\frac{a}{b} = 4$, then $T = \varnothing$.)

Let $F_1$ be the graph pictured in the top portion of
Figure~\ref{Fig:diamond2}.
Construct graph $F_2$ as follows:
take two copies of $F_1$,
the first with vertices labeled $1,\dots,9$,
as in Figure~\ref{Fig:diamond2},
the second with vertices similarly labeled $1',\dots,9'$.
Identify vertices $1$ and $1'$,
labeling the resulting vertex as $1$.
Similarly identify vertices $9$ and $9'$,
labeling the resulting vertex as $9$.
Lastly, add an edge joining vertices $8$ and $8'$.
The resulting graph $F_2$
is shown in Figure~\ref{Fig:diamond4}.

Define a color assignment $L$
such that $L(1) = X$ and $L(9) = Y,$
while
vertices $2$ through $8$ are assigned lists of colors
as shown in the bottom portion of
Figure~\ref{Fig:diamond2}.
Let vertices $2',\dots,8'$ be similarly assigned colors
so that vertex $2'$ has the same list as vertex $2$,
vertex $3'$ has the same list as $3$, and so on.

Then graph $F_2$ admits no $b$-fold $L$-coloring.
\end{lemma}

\begin{proof}
Suppose that $\varphi$ is a $b$-fold $L$-coloring of $F_2$.
By Lemma~\ref{L:planar-piece1},
more than half of the elements of $T$ lie in $\varphi(8)$.
Similarly,
more than half of the elements of $T$ lie in $\varphi(8')$.
So $\varphi(8)\cap\varphi(8') \ne \varnothing$.
But since vertices $8$ and $8'$ are adjacent,
this is impossible.\ggcendpf\end{proof}

When considering $(a:b)$-choosability,
Lemma~\ref{L:planar-piece2}
says we can forbid a specific
coloring of vertices $1$ and $9$ of $F_2$
with disjoint sets of colors.
Using this idea,
we can paste together copies of $F_2$
to construct a planar graph in which
all possible colorings of two vertices are forbidden.
We use this idea to prove Theorem~\ref{T:planar22-5}.

\begin{proof}[Proof of Theorem~\ref{T:planar22-5}]
Let $a$ and $b$ be positive integers
with $\frac{a}{b} < \frac{22}{5}$.
We construct a planar graph $G$
such that $G$ is not $(a:b)$-choosable.

If $\frac{a}{b} < 4$, then we may let $G = K_4$.
Suppose, therefore, that $4 \le \frac{a}{b} < \frac{22}{5}$.
Let
\[
q = \binom{a}{b} \binom{a-b}{b}.
\]

To construct graph $G,$
begin with $q$ copies of graph $F_2$,
pictured in Figure~\ref{Fig:diamond4}.
Identify all the $1$ vertices in these copies,
labeling the resulting vertex as $1$.
Similarly identify all the $9$ vertices,
labeling the resulting vertex as $9$.
Add an edge joining vertices $1$ and $9$.
Let $G$ be the resulting graph.

Graph $G$ is planar.
It remains to show that $G$ is not $(a:b)$-choosable.

Assign vertices $1$ and $9$ the same list of $a$ colors.
The number of ways these two vertices can be colored with
disjoint sets of size $b$ is $\binom{a}{b}\binom{a-b}{b} = q$.
Create an (arbitrary) correspondence between these colorings
and the $q$ copies of $F_2$.
For each possible coloring of vertices $1$ and $9$,
assign lists of colors to the $2,\dots,8$ and $2',\dots,8'$
vertices in the corresponding copy of $F_2$
so that Lemma~\ref{L:planar-piece2}
allows us to conclude that vertices $1$ and $9$
cannot be colored in this manner.

The result is an assignment $L$ of lists of colors
with $|L(v)| = a$ for all $v\in V(G)$,
such that $G$ admits no $b$-fold $L$-coloring,
since no coloring is possible for vertices $1$ and $9$.
Thus, $G$ is not $(a:b)$-choosable.\ggcendpf\end{proof}

We know that
if $\frac{a}{b}\ge 5$,
then every planar graph is $(a:b)$-choosable
(Theorem~\ref{T:planar5m-m}),
while, if $\frac{a}{b} < \frac{22}{5}$,
then there exists a planar graph
that is not $(a:b)$-choosable
(Theorem~\ref{T:planar22-5}).
From $\frac{22}{5} = 4\frac{2}{5}$ to $5$
is a sizable gap.
Can we close this gap?

\begin{conjecture}
  \label{J:planar-nice-choose}
There exists a real number $r$ such that
the following are equivalent
for positive integers $a$ and $b$.
\begin{enumerate}
\item Every planar graph is $(a:b)$-choosable.
\item $\displaystyle\frac{a}{b}\ge r$.\ggcnopf
\end{enumerate}
\end{conjecture}

If Conjecture~\ref{J:planar-nice-choose} is true,
then $\frac{22}{5} \le r \le 5$,
and it seems likely that
$r$ is either $\frac{9}{2}$ or $5$.
Zhu~\cite[Conjecture~4]{Zhu2017}
conjectured that every planar graph is $(9:2)$-choosable.
If Zhu's conjecture and Conjecture~\ref{J:planar-nice-choose}
both hold, then
$\frac{22}{5} = \frac{44}{10} \le r \le \frac{45}{10} = \frac{9}{2}$,
and most likely $r = \frac{9}{2}$.

As we will see in the next section,
we are able to ``close the gap''
for the larger class of $K_5$-minor-free graphs,
with $r = 5$;
see Theorem~\ref{T:k5mf-5m-m-iff}.

\medskip

What about planar graphs
with higher connectivity?

A number of examples have been found
of planar graphs that are not $4$-choosable.
The examples due to Voigt~\cite{Voi1993},
Gutner~\cite{Gut1996},
and Mirzakhani~\cite{Mir1996}
all have connectivity $3$.
Those due to Zhu~\cite{Zhu2017},
along with the examples constructed in this work,
have connectivity $2$;
however, by the well known fact that every edge-maximal
simple planar graph of order at least $4$ is $3$-connected,
we may add edges to these graphs
to obtain $3$-connected examples
without compromising their properties.

Increasing connectivity to $4$ is a different matter.
All of the examples mentioned are constructed by pasting together
small graphs along edges or triangular faces,
which naturally leads to connectivity at most $3$.
We ask whether examples with greater connectivity exist.

\begin{question}
Does there exist a $4$-connected planar graph
that is not $4$-choosable?\ggcnopf\end{question}

\begin{question}
For each positive integer $m$,
does there exist a $4$-connected planar graph
that is not $(4m:m)$-choosable?
What about ratios greater than $4$?\ggcnopf\end{question}

\section{$K_5$-Minor-Free Graphs} \label{S:k5minfree}

In this section we prove the following theorem.

\begin{theorem}
  \label{T:k5mf-5m-m-iff}
The following are equivalent
for positive integers $a$ and $b$.
\begin{enumerate}
\item Every $K_5$-minor-free graph is $(a:b)$-choosable.
\item $\displaystyle\frac{a}{b}\ge 5$.\ggcnopf
\end{enumerate}
\end{theorem}

Our proof uses a method similar to the proof of
Theorem~\ref{T:k5free-5}
by He, Miao, \& Shen~\cite[Thm.~2.1]{HMS2008},
along with examples based on a construction of
Bar\'{a}t, Joret, \& Wood~\cite[proof of Thm.~1]{BJW2011}.

We will make use of two previously known results.
The first, due to Tuza \& Voigt~\cite[Thm.~3.2]{TuVo1996},
deals with list multicoloring of planar graphs.

\begin{theorem}[Tuza \& Voigt 1996]
  \label{T:planar-5m-m}
Let $m$ be a positive integer.
Let $G$ be a plane near triangulation
with outer cycle $C.$
Suppose that $L$ is a list assignment for $G,$
with the following properties.
\begin{enumerate}
\item There are two adjacent vertices $u$ and $v$ of $C$
with disjoint lists of length $m$ each.
\item All the other vertices of $C$ have
(unrestricted)
lists of length $3m$.
\item All vertices of $G-C$ have lists of length $5m$.
\end{enumerate}
Then $G$ admits an $m$-fold $L$-coloring.\ggcnopf
\end{theorem}

The second previously known result
is a structure theorem for $K_5$-minor-free graphs
proven by Wagner~\cite{Wag1937}
and often called Wagner's Theorem.
As Wagner used terminology very different from ours,
we give a statement based on that
of Diestel~\cite[Thm.~8.3.4]{Die2000}.

\begin{figure}  
\centering
\includegraphics{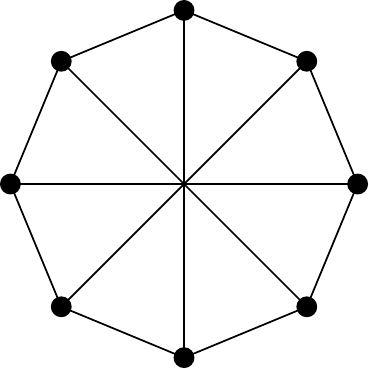}
\caption{Graph $M_8$, the $8$-vertex M\"{o}bius ladder,
also known as the Wagner Graph,
used in Theorem~\ref{T:wagner}
and the proof of Lemma~\ref{L:k5mf-5m-m}.}
\label{Fig:m8}
\end{figure}

\begin{theorem}[Wagner 1937]
  \label{T:wagner}
Let $G$
be an edge-maximal $K_5$-minor-free graph.
If $|V(G)|\ge 4$,
then $G$ can be constructed recursively,
by pasting along triangles and $K_2$s,
from plane triangulations and copies of the graph $M_8$---the
$8$-vertex M\"{o}bius ladder
or Wagner Graph,
shown in Figure~\ref{Fig:m8}.\ggcnopf
\end{theorem}

We begin with the following lemma,
which generalizes
a lemma of \v{S}krekovski~\cite[Lemma~2.2]{Skr1998}
and of
He, Miao, \& Shen~\cite[Lemma~2.1]{HMS2008}
to list multicoloring.
Our proof follows similar lines to that of
He, Miao, \& Shen.

\begin{lemma}
  \label{L:k5mf-5m-m}
Let $m$ be a positive integer.
Let $G$ be an edge-maximal $K_5$-minor-free graph,
and let $L$ be a list assignment for $G$ such that
$|L(v)|\ge 5m$ for each $v\in V(G)$.
Suppose that $H$ is a subgraph of $G$ isomorphic to $K_2$ or $K_3$,
and $\varphi$ is an $m$-fold $L$-coloring of $H.$
Then $\varphi$ can be extended to an $m$-fold $L$-coloring of $G.$
\end{lemma}

\begin{proof}
We proceed by induction on $|V(G)|$.
The result is immediate when $|V(G)|\le 3$.
Suppose $|V(G)|\ge 4$.

First suppose $G$ has a separating $K_2$ or $K_3$.
We may write $G = G_1\cup G_2$,
where $G_1\cap G_2$ is $K_2$ or $K_3$.
Then $H$ must be contained in either $G_1$ or $G_2$;
without loss of generality say $G_1$.
Apply the induction hypothesis to $G_1$,
with $H$ precolored by $\varphi$---adding
edges to $G_1$ as necessary to make it
edge-maximal $K_5$-minor-free.
Using the resulting multicoloring, precolor $G_1\cap G_2$
and apply the induction hypothesis to $G_2$---once
again adding edges as necessary.
Pasting the two resulting multicolorings together
gives the required $m$-fold $L$-coloring of $G.$

Now suppose $G\cong M_8$, the Wagner Graph;
see Figure~\ref{Fig:m8}.
Then the vertices of $G-H$ may be colored from their
lists in any order,
since every vertex in $G$ has degree $3$.

We are left with the case in which
$G$ is not the Wagner Graph
and has no separating $K_2$ or $K_3$.
By Theorem~\ref{T:wagner},
$G$ is a plane triangulation.
We may assume that $H$ is $K_3$;
otherwise choose a vertex that lies in a triangle
containing $H,$
color this vertex using colors in its list
that are not used on either vertex of $H,$
and add the new vertex to $H.$

Since $H$ is a triangle and is not separating,
we can
embed $G$ in the plane so that $H$ is the outer face.
Denote the vertices of $H$ by $x$, $y$, $z$.
Define a new list assignment $L'$ for $G$ as follows.
Let $L'(x) = \varphi(x)$,
$L'(y) = \varphi(y)$,
and
$L'(z) = \varphi(x)\cup\varphi(y)\cup\varphi(z)$.
For each vertex $v\in V(G)-\{x,y,z\}$,
let $L'(v)$ be a $5m$-element subset of $L(v)$.
Apply Theorem~\ref{T:planar-5m-m} to
obtain an $m$-fold $L'$-coloring of $G$
that extends $\varphi$.
This is the required $m$-fold $L$-coloring of $G.$\ggcendpf
\end{proof}

Next we prove a list-coloring property of the octahedron graph.

\begin{figure}  
\centering
\includegraphics{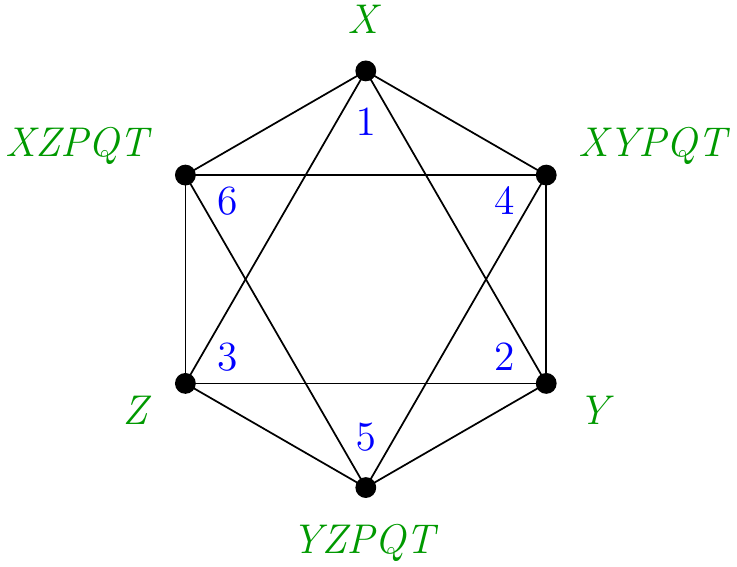}
\caption{Graph $N_8\cong K_{2,2,2}$, the octahedron graph,
with vertices labeled
and lists of colors shown,
used in Lemma~\ref{L:octa}
and the proof of Theorem~\ref{T:k5mf-5m-m-iff}.}
\label{Fig:octa}
\end{figure}

\begin{lemma}
  \label{L:octa}
Let $a$ and $b$ be positive integers with
$4\le \frac{a}{b} < 5$.
Let $X,$ $Y,$ $Z,$ $P,$ $Q,$ and $T$ be pairwise disjoint
lists of colors,
such that $X,$ $Y,$ $Z,$ $P,$ and $Q$ have size $b$,
while $T$ has size $a-4b$.
(Note that $a-4b \ge 0$;
if $\frac{a}{b} = 4$, then $T = \varnothing$.)

Let $N_8\cong K_{2,2,2}$ be the octahedron graph
with list assignment $L$ shown in Figure~\ref{Fig:octa}.

Then graph $N_8$ admits no $b$-fold $L$-coloring.
\end{lemma}

\begin{proof}
In a $b$-fold $L$-coloring $\varphi$ of $N_8$,
we must have $\varphi(1) = X,$
$\varphi(2) = Y,$
and $\varphi(3) = Z.$
So $\varphi(4)$, $\varphi(5)$, and $\varphi(6)$
are pairwise disjoint subsets of $\mathit{PQT},$
each of size $b$,
and thus $|\varphi(4)\cup\varphi(5)\cup\varphi(6)| = 3b$.
However, $|\mathit{PQT}| < 3b$,
so no such coloring can exist.\ggcendpf
\end{proof}

When considering $(a:b)$-choosability,
Lemma~\ref{L:octa}
says we can forbid a specific
coloring of vertices $1$, $2$, and $3$ of the octahedron graph $N_8$
with pairwise disjoint sets of colors.
Using this idea,
we can paste together copies of $N_8$
to construct a $K_5$-minor-free graph in which
all possible colorings of three vertices are forbidden.
We use this idea in the proof of Theorem~\ref{T:k5mf-5m-m-iff}.

\begin{proof}[Proof of Theorem~\ref{T:k5mf-5m-m-iff}]
(ii) $\Longrightarrow$ (i).
Let $a$ and $b$ be positive integers with $\frac{a}{b} \ge 5$.
Let $G$ be a $K_5$-minor-free graph.
We may assume that $G$ is edge-maximal $K_5$-minor-free;
otherwise add edges until it is.
Let $L$ be a list assignment of $G$ such that
$|L(v)| = a$ for each $v\in V(G)$.
It suffices to show that $G$ admits a $b$-fold $L$-coloring.

If $|V(G)| < 2$, then the result is immediate.
If $|V(G)| \ge 2$,
then let $H$ be the subgraph induced
by two (arbitrary) adjacent vertices.
Let $m = b$,
and let $\varphi$ be an $m$-fold $L$-coloring of $H.$
Applying Lemma~\ref{L:k5mf-5m-m},
we obtain the required $b$-fold $L$-coloring of $G.$

\medskip
\noindent
(i) $\Longrightarrow$ (ii).
Let $a$ and $b$ be positive integers
with $\frac{a}{b} < 5$.
We construct a $K_5$-minor-free graph $G$
such that $G$ is not $(a:b)$-choosable.

If $\frac{a}{b} < 4$, then we may let $G = K_4$.
Suppose, therefore, that $4\le \frac{a}{b} < 5$.
Let
\[
q = \binom{a}{b} \binom{a-b}{b} \binom{a-2b}{b}.
\]

To construct graph $G$,
begin with $q$ copies of the octahedron graph $N_8$,
pictured in Figure~\ref{Fig:octa}.
In each copy, there is a triangle having vertices
labeled $1$, $2$, $3$.
Paste along all of these triangles to obtain $G,$
so that $G$ has just one vertex labeled $1$,
and similarly for $2$ and $3$.
See Figure~\ref{Fig:multiocta} for an illustration
of graph $G.$
(The construction of graph $G$
is a variation on a construction of
Bar\'{a}t, Joret, \& Wood~\cite[proof of Thm.~1]{BJW2011}.)

Graph $G$ has no $K_5$ minor.
It remains to show that $G$ is not $(a:b)$-choosable.

Assign vertices $1$, $2$, and $3$ the same list of $a$ colors.
The number of ways these three vertices can be colored with
pairwise disjoint sets of size $b$ is
$\binom{a}{b}\binom{a-b}{b}\binom{a-2b}{b} = q$.
Create an (arbitrary) correspondence between these colorings
and the $q$ copies of $N_8$.
For each possible coloring of vertices $1$, $2$, $3$,
assign lists of colors to the other $3$ vertices
of the corresponding copy of $N_8$
so that Lemma~\ref{L:octa} allows us to conclude
that vertices $1$, $2$, $3$ cannot be colored in this manner.

The result is an assignment $L$ of lists of colors
with $|L(v)| = a$ for all $v\in V(G)$,
such that $G$ admits no $b$-fold $L$-coloring,
since no coloring is possible for vertices $1$, $2$, $3$.
Thus, $G$ is not $(a:b)$-choosable.\ggcendpf\end{proof}

\begin{figure}  
\centering
\includegraphics{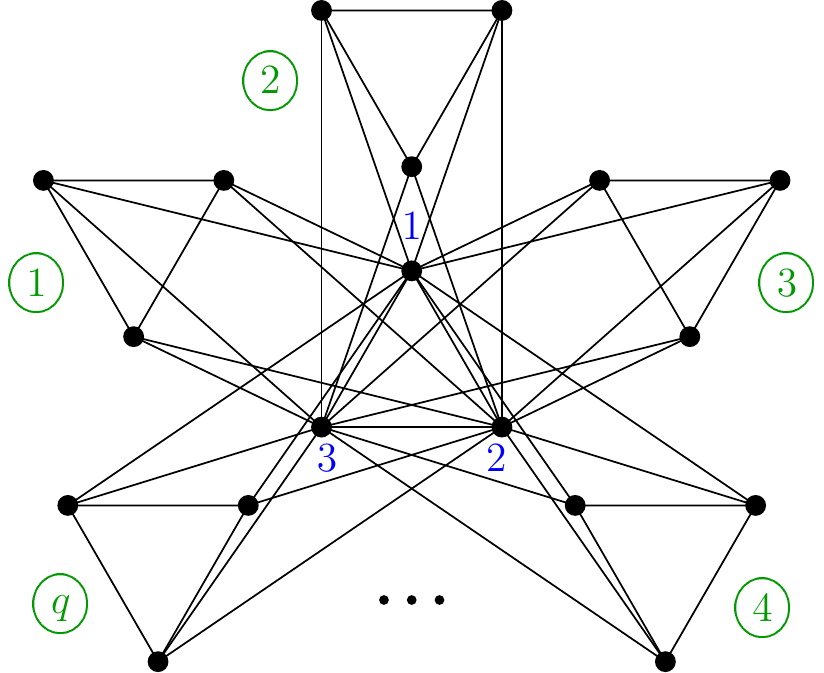}
\caption{Graph $G$ from the (i) $\Longrightarrow$ (ii) part
of the proof of Theorem~\ref{T:k5mf-5m-m-iff}:
$q$ copies of the octahedron graph $N_8$
(see Figure~\ref{Fig:octa})
pasted together,
with the $1,2,3$ triangles identified.
Each copy of $N_8$
is labeled with a circled number from $1$ to $q$.}
\label{Fig:multiocta}
\end{figure}

What about forbidding clique minors of orders other than $5$?
Results similar to Theorem~\ref{T:k5mf-5m-m-iff}
can easily be proven for smaller clique minors.

\begin{theorem}
  \label{T:k234mf-iff}
Let $t$ be a positive integer with $2\le t\le 4$.
Then the following are equivalent
for positive integers $a$ and $b$.
\begin{enumerate}
\item Every $K_t$-minor-free graph is $(a:b)$-choosable.
\item $\displaystyle\frac{a}{b}\ge t-1$.\ggcnopf
\end{enumerate}
\end{theorem}

\begin{proof}
(i) $\Longrightarrow$ (ii).
Graph
$K_{t-1}$ is a $K_t$-minor-free graph
that is not $(a:b)$-choosable
for any $a$, $b$ with $\frac{a}{b} < t-1$.

\medskip

\noindent
(ii) $\Longrightarrow$ (i).
For $2\le t\le 4$,
every $K_t$-minor-free graph with order at least $1$
has a vertex of degree at most $t-2$.
In particular,
a $K_2$-minor-free graph is an edgeless graph,
which of course has a vertex of degree at most $0$.
A $K_3$-minor-free graph is a forest,
which must have a vertex of degree at most $1$.
And it follows
from a proof of Dirac~\cite[p.~87]{Dir1952}
(see also Duffin~\cite[Thm.~1, Corollary~4]{Duf1965})
that a $K_4$-minor-free graph
must have a vertex of degree at most $2$.

By a simple inductive argument, then,
for $2\le t\le 4$,
every $K_t$-minor-free graph is $(a:b)$-choosable
for all $a$, $b$ with $\frac{a}{b}\ge t-1$
(see Tuza \& Voigt~\cite[Thm.~2.1]{TuVo1996}).\ggcendpf
\end{proof}

Theorem~\ref{T:k234mf-iff} is not really new;
it simply restates well known ideas.
But it is a result in our desired form
(see Section~\ref{S:intro}),
and considering it together with Theorem~\ref{T:k5mf-5m-m-iff}
is suggestive.
We ask whether results of this kind hold
for $K_t$-minor-free graphs
for larger values of $t$.

\begin{question}
  \label{Q:kALLmf-iff}
Is it true that
for each integer $t\ge 2$,
there exists a real number $r_t$
such that the following are equivalent
for positive integers $a$ and $b$?
\begin{enumerate}
\item Every $K_t$-minor-free graph
is $(a:b)$-choosable.
\item $\displaystyle\frac{a}{b}\ge r_t$.\ggcnopf
\end{enumerate}
\end{question}

Question~\ref{Q:kALLmf-iff}
has an affirmative answer
for $2\le t\le 5$,
by Theorems~\ref{T:k5mf-5m-m-iff}
and
\ref{T:k234mf-iff}.
It remains open for larger values of $t$.


\end{document}